\documentclass[12pt]{article}
\usepackage{amssymb}
\usepackage{amsthm}
\usepackage{amsmath, mathrsfs}
\usepackage{amsfonts}
\newcommand{\be}{\begin{equation}}
\newcommand{\ee}{\end{equation}}
\newcommand{\bea}{\begin{eqnarray*}}
\newcommand{\eea}{\end{eqnarray*}}
\newcommand{\ba}{\begin{array}}
\newcommand{\ea}{\end{array}}
\newcommand{\bi}{\begin{itemize}}
\newcommand{\ei}{\end{itemize}}
\newcommand{\bc}{\begin{center}}
\newcommand{\ec}{\end{center}}
\newcommand{\bfr}{\begin{flushright}}
\newcommand{\efr}{\end{flushright}}

\newtheorem{Pa}{Paper}[section]
\newtheorem{theorem}[Pa]{{\bf Theorem}}

\newtheorem{remark}[Pa]{{\bf Remark}}

\begin{document}

\title{Semi-discrete Gr\"uss-Voronovskaya-type and Gr\"uss-type estimates for Bernstein-Kantorovich polynomials}
\author{Sorin G. Gal\\
%EndAName
University of Oradea\\
Department of Mathematics and Computer Science\\
Str. Universitatii Nr. 1\\
410087 Oradea, Romania\\
e-mail : galso@uoradea.ro}
\date{}
\maketitle

{\bf Abstract.} The aim of this note is to prove a semi-discrete Gr\"uss-Voronovskaya-type
estimate for Bernstein-Kantorovich polynomials. Also, as a consequence, a perturbed Gr\"uss-type estimate is obtained.

{\bf Keywords.} Bernstein-Kantorovich polynomials, semi-discrete Gr\"uss-Voronovskaya-type estimate, perturbed Gr\"uss-type estimate, modulus of continuity.

\textbf{AMS 2000 Mathematics Subject Classification}: 41A36, 41A25, 41A60.

\section{Introduction}

A classical result in approximation theory is the asymptotic qualitative result of Voronovskaya for Bernstein polynomials in \cite{Voronov}. It was generalized by Bernstein in \cite{Bern} and then it was extended to positive and linear operators by Mamedov in \cite{Mamedov}. Also, quantitative estimates of Mamedov's result were obtained in terms of the least concave majorant and a $K$-functional by Gonska in \cite{Gonska2} and by Gavrea-Ivan in \cite{Gav2}.

Another classical result is the well-known Gr\"uss inequality for positive linear functionals $L : C[0,1] \to \mathbb{R}$. This inequality gives an upper bound for the generalized Chebyshev functional
$$
T(f,g) := L(f\cdot g) - L(f) \cdot L(g) , \quad f,g \in C[0,1].
$$
For positive and linear operators $H : C[0,1] \to C[0,1]$ reproducing constant functions, this was investigated for the first time in \cite{AGR}, then obtaining in \cite {GoRaRu} the estimate
$$
|H(fg;x) - H(f;x) \cdot H(g;x)|
$$
$$
\le \frac14 \cdot \tilde{\omega}_1 (f; 2 \cdot \sqrt{H(e_2;x)-H(e_1;x)^2}) \cdot
\tilde{\omega}_1 (g;2\cdot \sqrt{H(e_2;x) - H(e_1;x)^2})
$$
where $\tilde{\omega}_1$ is the least concave majorant of $\omega_1$ and $e_i (x) = x^i$ for $x \in [0,1]$.

A mixture between the above two classical results are the so-called Gr\"uss-Voronovskaya-type results obtained for the first time in the paper \cite{Gal-Gonska} for Bernstein and  P\v{a}lt\v{a}nea operators.

On the other hand, in the very recent paper \cite{Gal-Discrete}, we generalized the asymptotic quantitative Voronovskaya-type results, by obtaining semi-discrete quantitative Voronovskaya-type results for general positive and linear operators.

The main goal of this short note is to use the result in \cite{Gal-Discrete} to obtain in Section 2 a semi-discrete Gr\"uss-Voronovskaya-type results for the Bernstein-Kantorovich polynomials. Also, as  a consequence, we easily obtain a perturbed Gr\"uss-type estimate for the same polynomials.

\section{Semi-discrete Gr\"uss-Voronovskaya-type estimate}

The main result is the following semi-discrete Gr\"uss-Voronovskaya-type estimate, for the Bernstein-Kantorovich polynomials given by the formula (see \cite{Kant})
$$K_{n}(f)(x)=\sum_{k=0}^{n}p_{n, k}(x)\cdot (n+1)\int_{k/(n+1)}^{(k+1)/(n+1)}f(t) d t,$$
where $p_{n, k}(x){n\choose k}x^{k}(1-x)^{n-k}$ and $f:[0, 1]\to \mathbb{R}$ is Riemann (or Lebesgue) integrable on $[0, 1]$.

Also, for $n\in \mathbb{N}$ and $,y\in [0, 1]$, let us denote
\begin{equation*}
\begin{split}
E_{n}(x, y)&=\frac{1}{(n+1)^{2}}\cdot \left (x(1-x)(n-1) + \frac{1}{3}\right )+(x-y)\frac{1-2 x}{2(n+1)}\\
F_{n}(x)&=\frac{1}{(n+1)^{2}}\cdot \left (x(1-x)(n-1) + \frac{1}{3}\right ).
\end{split}
\end{equation*}
Notice that clearly we have $|E_{n}(x, y)|={\mathcal{O}}\left (\frac{1}{n}\right )$ and $|F_{n}(x)|={\mathcal{O}}\left (\frac{1}{n}\right )$, uniformly with respect to $x, y\in [0, 1]$.
\begin{theorem}\label{thm4.1} For all $f, g\in C^{2}[0, 1]$, $n\in \mathbb{N}$ and $x, y\in [0, 1]$, $x\not=y$ we have
\begin{equation*}
\begin{split}
&\left |K_{n}(f g)(x)-K_{n}(f)(x)\cdot K_{n}(g)(x)+(x-y)\cdot \frac{1-2x}{2(n+1)}\left ([x, y; f]\cdot [x, y; g]-f^{\prime}(x)g^{\prime}(x)\right )\right .\\
-&F_{n}(x)\cdot f^{\prime}(x)\cdot g^{\prime}(x)|\\
\le& \left [\frac{1}{(n+1)^{2}}(x(1-x)(n-1)+ 1/3)+|x-y|\cdot \frac{1}{\sqrt{3}\sqrt{n+1}}\right ]\\
\cdot &\left [\omega_{1}\left ((f g)^{\prime \prime}; |x-y|+\frac{2 \sqrt{6}}{\sqrt{n+1}}\right )
+\|g\|\omega_{1}\left (f^{\prime \prime}; |x-y|+\frac{2 \sqrt{6}}{\sqrt{n+1}}\right )\right .\\
&\left . + \|f\|\omega_{1}\left (g^{\prime \prime}; |x-y|+\frac{2 \sqrt{6}}{\sqrt{n+1}}\right )\right ]+|K_{n}(f)(x)-f(x)|\cdot |K_{n}(g)(x)-g(x)|,
\end{split}
\end{equation*}
where $[x, y; f]=\frac{f(x)-f(y)}{x-y}$, $\omega_{1}(f; \delta):=\sup\{|f(x)-f(y)|; x, y\in [0, 1], |x-y|\le \delta\}$ and $\|f\|$ denotes the uniform norm of $f$.
\end{theorem}
\begin{proof}
Supposing that $f, g\in C^{2}[0, 1]$ and using Corollary 4.1  in \cite{Gal-Discrete}, we obtain
\begin{equation*}
\begin{split}
&\left |K_{n}(f g)(x)-K_{n}(f)(x)\cdot K_{n}(g)(x)+(x-y)\cdot \frac{1-2x}{2(n+1)}\left ([x, y; f]\cdot [x, y; g]-f^{\prime}(x)g^{\prime}(x)\right )\right .\\
-&F_{n}(x)\cdot f^{\prime}(x)\cdot g^{\prime}(x)|\\
&=\left |\left (K_{n}(f g)(x)-f(x)g(x)-\frac{1-2 x}{2(n+1)}\cdot [x, y; f g]\right )-E_{n}(x, y)\cdot \frac{(f(x) g(x))^{\prime \prime}}{2}\right .\\
&\left . -g(x)\left [\left (K_{n}(f)(x)-f(x)-\frac{1-2 x}{2(n+1)}\cdot [x, y; f]\right )-E_{n}(x, y)\cdot \frac{f^{\prime \prime}(x)}{2}\right ]\right .\\
&\left . -f(x)\left [\left (K_{n}(g)(x)-g(x)-\frac{1-2 x}{2(n+1)}\cdot [x, y; g]\right )-E_{n}(x, y)\cdot \frac{g^{\prime \prime}(x)}{2}\right ]\right .\\
&+\left . [K_{n}(f)(x)-f(x)]\cdot [g(x)-K_{n}(g)(x)]\right |\\
&\le \left [\frac{1}{(n+1)^{2}}(x(1-x)(n-1)+ 1/3)+|x-y|\cdot \frac{1}{\sqrt{3}\sqrt{n+1}}\right ]\\
&\cdot \left [\omega_{1}\left ((f g)^{\prime \prime}; |x-y|+\frac{2 \sqrt{6}}{\sqrt{n+1}}\right )
+\|g\|\cdot \omega_{1}\left (f^{\prime \prime}; |x-y|+\frac{2 \sqrt{6}}{\sqrt{n+1}}\right )\right .\\
&+ \left .\|f\|\cdot
\omega_{1}\left (g^{\prime \prime}; |x-y|+\frac{2 \sqrt{6}}{\sqrt{n+1}}\right )\right ]+|K_{n}(f)(x)-f(x)|\cdot |K_{n}(g)(x)-g(x)|,
\end{split}
\end{equation*}
which is exactly the estimate in the statement.
\end{proof}
\begin{remark} Let $f, g\in C^{3}[0, 1]$. Firstly, take $y\to x$, multiply by $n$ both members in the estimate in Theorem \ref{thm4.1} and
use the estimate in \cite{AH}, page 849, line 7 from below
$$|K_{n}(h)(x)-h(x)|\le \frac{1}{2 n}\|h^{\prime}\|+\frac{8}{9 n}\|h^{\prime\prime}\|, x\in [0, 1], n\in \mathbb{N}, h\in C^{2}[0, 1].$$
Then, since
$$n\cdot F_{n}(x)=\frac{n(n-1)}{(n+1)^{2}}x(1-x)+\frac{n}{3(n+1)^{2}},$$
by using the estimate in Theorem \ref{thm4.1}, we easily obtain
$$\|n[K_{n}(f g)-K_{n}(f)\cdot K_{n}(g)]-e_{1}(1-e_{1})f^{\prime}g^{\prime}\|={\mathcal{O}}\left (\frac{1}{\sqrt{n}}\right ),$$
thus recapturing the order of approximation in the classical Gr\"uss-Voronovskaya-type estimate given by Theorem 5.1 in \cite{AH}.
\end{remark}
\begin{remark} Since obviously
$$\left |K_{n}(f g)(x)-K_{n}(f)(x)\cdot K_{n}(g)(x)+(x-y)\cdot \frac{1-2x}{2(n+1)}\left ([x, y; f]\cdot [x, y; g]-f^{\prime}(x)g^{\prime}(x)\right )\right |$$
\begin{equation*}
\begin{split}
&\le \left |K_{n}(f g)(x)-K_{n}(f)(x)\cdot K_{n}(g)(x)+(x-y)\cdot \frac{1-2x}{2(n+1)}\left ([x, y; f]\cdot [x, y; g]-f^{\prime}(x)g^{\prime}(x)\right )\right .\\
&\left .-F_{n}(x)\cdot f^{\prime}(x)\cdot g^{\prime}(x)\right |+|F_{n}(x)|\cdot \left |f^{\prime}(x)\cdot g^{\prime}(x)\right |,
\end{split}
\end{equation*}
by Theorem 2.1, for all $f, g\in C^{2}[0, 1]$, $n\in \mathbb{N}$, $x, y\in [0, 1]$, $x\not=y$,  we immediately get the following estimate
\begin{equation*}
\begin{split}
&\left |K_{n}(f g)(x)-K_{n}(f)(x)\cdot K_{n}(g)(x)+(x-y)\cdot \frac{1-2x}{2(n+1)}\left ([x, y; f]\cdot [x, y; g]-f^{\prime}(x)g^{\prime}(x)\right )\right |\\
&\le \left [\frac{1}{(n+1)^{2}}(x(1-x)(n-1)+ 1/3)+|x-y|\cdot \frac{1}{\sqrt{3}\sqrt{n+1}}\right ]\\
&\cdot \left [\omega_{1}\left ((f g)^{\prime \prime}; |x-y|+\frac{2 \sqrt{6}}{\sqrt{n+1}}\right )
+\|g\|\cdot \omega_{1}\left (f^{\prime \prime}; |x-y|+\frac{2 \sqrt{6}}{\sqrt{n+1}}\right )\right .\\
&+ \left .\|f\|\cdot
\omega_{1}\left (g^{\prime \prime}; |x-y|+\frac{2 \sqrt{6}}{\sqrt{n+1}}\right )\right ]+|K_{n}(f)(x)-f(x)|\cdot |K_{n}(g)(x)-g(x)|\\
&+|F_{n}(x)|\cdot |f^{\prime}(x)g^{\prime}(x)|,
\end{split}
\end{equation*}
which can be considered as a "perturbed" (discrete) Gr\"uss-type estimate since for $y$ sufficiently close to $x$, the left-hand side of the above inequality, becomes sufficiently close to
$$|K_{n}(f g)(x)-K_{n}(f)(x)\cdot K_{n}(g)(x)|.$$

Now, if above  we take $y\to x$, then since $|F_{n}(x)|={\mathcal{O}}\left (\frac{1}{n}\right )$,
we immediately get
$$\|K_{n}(f g)-K_{n}(f)\cdot K_{n}(g)\|={\mathcal{O}}\left (\frac{1}{n}\right ),$$
which is the same order which can be obtained for the classical Gr\"uss-type estimate in terms of the least concave majorant of the modulus of continuity expressed by Theorem 4.2 in \cite{AH} for $f, g\in C^{2}[0, 1]$.
\end{remark}
\begin{remark} The results in this note suggest that based on other semi-discrete Vo\-ro\-nov\-ska\-ya-type results in \cite{Gal-Discrete}, to get for the Bernstein-Kantorovich polynomials other semi-discrete estimates of Gr\"uss-Voronovskaya-type and of Gr\"uss-type. Also, similar results can be obtained for other positive and linear operators too.
\end{remark}

\end{document}